\documentclass[12pt]{amsart}

\usepackage[top=1.4in,bottom=1.275in,left=1.2in,right=1.2in]{geometry}

\usepackage[utf8]{inputenc}
\usepackage[T1]{fontenc}
\usepackage[english]{babel}\usepackage[activate={true,nocompatibility},final,tracking=true,factor=1100,stretch=10,shrink=10,]{microtype}\microtypecontext{spacing=nonfrench}
\usepackage{placeins}
\usepackage{comment}

\PassOptionsToPackage{usenames,dvipsnames}{xcolor}

\usepackage{amsmath}\usepackage{amssymb}\usepackage{amsthm}\usepackage{amscd}

\usepackage{bm}\usepackage{mathrsfs}\usepackage{mathtools}\usepackage{centernot}
\usepackage{xfrac}
\usepackage[only,llbracket,rrbracket]{stmaryrd}\SetSymbolFont{stmry}{bold}{U}{stmry}{m}{n}
\usepackage{calligra}
\DeclareMathAlphabet{\mathcalligra}{T1}{calligra}{m}{n}
\DeclareFontShape{T1}{calligra}{m}{n}{<->s*[1.1]callig15}{}
\DeclareFontFamily{OT1}{pzc}{}
\DeclareFontShape{OT1}{pzc}{m}{it}{<-> s * [1.0] pzcmi7t}{}
\DeclareMathAlphabet{\mathpzc}{OT1}{pzc}{m}{it}

\usepackage[mathscr]{eucal}
\usepackage{tikz}
\usepackage{pgf}
\usepackage{pgfplots}
\usepackage{pgfplotstable}
\pgfplotsset{compat=newest}

\usepackage{graphicx}\usepackage[font=small,figurewithin=section]{caption}\usepackage[labelformat=simple]{subcaption}
\usepackage{float}\usepackage[hypertexnames=false]{hyperref}\usepackage[capitalize,nameinlink,noabbrev]{cleveref}

\crefname{equation}{}{}
\usepackage{autonum}\usepackage{tikz}
\usepackage{verbatim}

\definecolor{CeruleanRef}{RGB}{12,127,172}
\hypersetup{colorlinks=true,allcolors=CeruleanRef}
\hypersetup{pdfstartview=FitB,pdfpagemode=UseNone}

\makeatletter
\renewcommand*{\eqref}[1]{  \hyperref[{#1}]{\textup{\tagform@{\ref*{#1}}}}}
\makeatother

\newtheorem{theorem}{Theorem}

\newtheorem{corollary}[theorem]{Corollary}
\newtheorem{proposition}[theorem]{Proposition}

\theoremstyle{definition}

\theoremstyle{remark}
\newtheorem{remark}[theorem]{Remark}

\usepackage{todonotes}

\let\inf\relax \DeclareMathOperator*\inf{\vphantom{p}inf}
\let\min\relax \DeclareMathOperator*\min{\vphantom{p}min}
 
\let\subset\relax \DeclareMathOperator{\subset}{\subseteq}
\let\supset\relax \DeclareMathOperator{\supset}{\supseteq}

\def\quotient#1#2{    \raise0.1pt\hbox{$#1$}/\lower0.6pt\hbox{$#2$}}

\makeatletter
\newcommand{\normm}{\@ifstar\@normms\@normm}
\newcommand{\@normms}[1]{  \left|\mkern-1.5mu\left|\mkern-1.5mu\left|
   #1
  \right|\mkern-1.5mu\right|\mkern-1.5mu\right|
}
\newcommand{\@normm}[2][]{  \mathopen{#1|\mkern-1.5mu#1|\mkern-1.5mu#1|}
  #2
  \mathclose{#1|\mkern-1.5mu#1|\mkern-1.5mu#1|}
}
\makeatother

\newsavebox\CBox

\makeatletter
\newcommand{\leqnomode}{\tagsleft@true\let\veqno\@@leqno}
\newcommand{\reqnomode}{\tagsleft@false\let\veqno\@@eqno}
\makeatother

\newcommand{\R}{\mathbb{R}}

 \newcommand{\mesh}{\mcT_{h}}   \DeclareMathOperator{\dd}{d} \newcommand{\bdry}{\partial}

\newcommand{\bzeta}{\bm{\zeta}}

\newcommand{\blambda}{\bm{\lambda}}

\newcommand{\bnu}{\bm{\nu}}

\newcommand{\bsigma}{\bm{\sigma}}
\newcommand{\btau}{\bm{\tau}}

\renewcommand\Omega{\varOmega}
\renewcommand\Gamma{\varGamma}
\renewcommand\Pi{\varPi}

\DeclareMathOperator{\grad}{grad}

\let\div\relax \DeclareMathOperator{\div}{div}

\makeatletter
\newcommand{\dashto}[1][2pt]{
  \settowidth{\@tempdima}{${}\rightarrow{}$}
  \makebox[\@tempdima]{${}\rightarrow{}$}  \makebox[-\@tempdima]{\hspace{-0.1\@tempdima}\color{white}\rule[0.5ex]{#1}{1pt}}  \makebox[\@tempdima]{}  }
\makeatother

\newcommand{\upperone}[2]{\raisebox{-0.2\height}{$#1{}^1\!$}}
\newcommand{\upone}{\mathpalette\upperone\relax}
\newcommand{\lowertwo}[2]{\raisebox{-0.1\height}{$#1{}_{\!2}$}}

\newcommand{\lowtwo}{\mathpalette\lowertwo\relax}

\newcommand{\frachalf}[2]{\raisebox{0.0\height}{$#1{}\upone/\lowtwo$}}

\newcommand{\onehalf}{\mathpalette\frachalf\relax}

\newcommand{\mcL}{\mathcal{L}}

\newcommand{\mcT}{\mathcal{T}}

\newcommand{\bmb}{\bm{b}}

\newcommand{\bmf}{\bm{f}}

\newcommand{\bmu}{\bm{u}}
\newcommand{\bmv}{\bm{v}}

\newcommand{\bmA}{\bm{A}}

\newcommand{\bmI}{\bm{I}}

\newcommand{\bmV}{\bm{V}}

\newcommand{\bfx}{\mathbf{x}}

\DeclareSymbolFont{sfletters}{OML}{cmbrm}{m}{it}
\DeclareMathSymbol{\sfalpha}{\mathord}{sfletters}{"0B}
\DeclareMathSymbol{\sfbeta}{\mathord}{sfletters}{"0C}
\DeclareMathSymbol{\sfgamma}{\mathord}{sfletters}{"0D}
\DeclareMathSymbol{\sfdelta}{\mathord}{sfletters}{"0E}
\DeclareMathSymbol{\sfepsilon}{\mathord}{sfletters}{"0F}
\DeclareMathSymbol{\sfzeta}{\mathord}{sfletters}{"10}
\DeclareMathSymbol{\sfeta}{\mathord}{sfletters}{"11}
\DeclareMathSymbol{\sftheta}{\mathord}{sfletters}{"12}
\DeclareMathSymbol{\sfiota}{\mathord}{sfletters}{"13}
\DeclareMathSymbol{\sfkappa}{\mathord}{sfletters}{"14}
\DeclareMathSymbol{\sflambda}{\mathord}{sfletters}{"15}
\DeclareMathSymbol{\sfmu}{\mathord}{sfletters}{"16}
\DeclareMathSymbol{\sfnu}{\mathord}{sfletters}{"17}
\DeclareMathSymbol{\sfxi}{\mathord}{sfletters}{"18}
\DeclareMathSymbol{\sfpi}{\mathord}{sfletters}{"19}
\DeclareMathSymbol{\sfrho}{\mathord}{sfletters}{"1A}
\DeclareMathSymbol{\sfsigma}{\mathord}{sfletters}{"1B}
\DeclareMathSymbol{\sftau}{\mathord}{sfletters}{"1C}
\DeclareMathSymbol{\sfupsilon}{\mathord}{sfletters}{"1D}
\DeclareMathSymbol{\sfphi}{\mathord}{sfletters}{"1E}
\DeclareMathSymbol{\sfchi}{\mathord}{sfletters}{"1F}
\DeclareMathSymbol{\sfpsi}{\mathord}{sfletters}{"20}
\DeclareMathSymbol{\sfomega}{\mathord}{sfletters}{"21}
\DeclareMathSymbol{\sfvarepsilon}{\mathord}{sfletters}{"22}
\DeclareMathSymbol{\sfvartheta}{\mathord}{sfletters}{"23}
\DeclareMathSymbol{\sfvarpi}{\mathord}{sfletters}{"24}
\DeclareMathSymbol{\sfvarrho}{\mathord}{sfletters}{"25}
\DeclareMathSymbol{\sfvarsigma}{\mathord}{sfletters}{"26}
\DeclareMathSymbol{\sfvarphi}{\mathord}{sfletters}{"27}
\DeclareSymbolFont{bsfletters}{OML}{cmbrm}{b}{it}
\DeclareMathSymbol{\bsfalpha}{\mathord}{sfletters}{"0B}
\DeclareMathSymbol{\bsfbeta}{\mathord}{sfletters}{"0C}
\DeclareMathSymbol{\bsfgamma}{\mathord}{sfletters}{"0D}
\DeclareMathSymbol{\bsfdelta}{\mathord}{sfletters}{"0E}
\DeclareMathSymbol{\bsfepsilon}{\mathord}{sfletters}{"0F}
\DeclareMathSymbol{\bsfzeta}{\mathord}{sfletters}{"10}
\DeclareMathSymbol{\bsfeta}{\mathord}{sfletters}{"11}
\DeclareMathSymbol{\bsftheta}{\mathord}{sfletters}{"12}
\DeclareMathSymbol{\bsfiota}{\mathord}{sfletters}{"13}
\DeclareMathSymbol{\bsfkappa}{\mathord}{sfletters}{"14}
\DeclareMathSymbol{\bsflambda}{\mathord}{sfletters}{"15}
\DeclareMathSymbol{\bsfmu}{\mathord}{sfletters}{"16}
\DeclareMathSymbol{\bsfnu}{\mathord}{sfletters}{"17}
\DeclareMathSymbol{\bsfxi}{\mathord}{sfletters}{"18}
\DeclareMathSymbol{\bsfpi}{\mathord}{sfletters}{"19}
\DeclareMathSymbol{\bsfrho}{\mathord}{sfletters}{"1A}
\DeclareMathSymbol{\bsfsigma}{\mathord}{sfletters}{"1B}
\DeclareMathSymbol{\bsftau}{\mathord}{sfletters}{"1C}
\DeclareMathSymbol{\bsfupsilon}{\mathord}{sfletters}{"1D}
\DeclareMathSymbol{\bsfphi}{\mathord}{sfletters}{"1E}
\DeclareMathSymbol{\bsfchi}{\mathord}{sfletters}{"1F}
\DeclareMathSymbol{\bsfpsi}{\mathord}{sfletters}{"20}
\DeclareMathSymbol{\bsfomega}{\mathord}{sfletters}{"21}
\DeclareMathSymbol{\bsfvarepsilon}{\mathord}{sfletters}{"22}
\DeclareMathSymbol{\bsfvartheta}{\mathord}{sfletters}{"23}
\DeclareMathSymbol{\bsfvarpi}{\mathord}{sfletters}{"24}
\DeclareMathSymbol{\bsfvarrho}{\mathord}{sfletters}{"25}
\DeclareMathSymbol{\bsfvarsigma}{\mathord}{sfletters}{"26}
\DeclareMathSymbol{\bsfvarphi}{\mathord}{sfletters}{"27}

\newsavebox{\foobox}

\newcommand{\LL}{\mathcal{L}}

\newcommand{\revised}[1]{#1}

\title[A priori error analysis of high-order LL* methods]{A priori error analysis of high-order LL* (FOSLL*) finite element methods}
\author[Keith]{Brendan Keith}
\address[Keith]{
  Lawrence Livermore National Laboratory,
  Livermore, CA 94550
  }
\email{keith10@llnl.gov}


 \date{}

\begin{document}

\begin{abstract}
  A number of non-standard finite element methods have been proposed in recent years, each of which derives from a specific class of PDE-constrained norm minimization problems.
  The most notable examples are $\mcL\mcL^\ast$ methods.
  In this work, we argue that all high-order methods in this class should be expected to deliver substandard uniform $h$-refinement convergence rates.
  In fact, one may not even see rates proportional to the polynomial order $p>1$ when the exact solution is a constant function.
      We show that the convergence rate is limited by the regularity of an extraneous Lagrange multiplier variable which naturally appears via a saddle-point analysis.
    In turn, limited convergence rates appear because the regularity of this Lagrange multiplier is determined, in part, by the geometry of the domain.
      Numerical experiments support our conclusions.
          \end{abstract}

\maketitle

\section{Introduction} \label{sec:introduction}

A sustained scientific interest in finite element methods which are intrinsically stable has delivered of a large number of new non-standard finite element methods throughout the last decade \cite{dahmen2012adaptive,cohen2012adaptivity,chan2014dual,BACUTA20152920,bacuta2017saddle,Calo2018automatic,kalchev2018mixed,muga2019discrete,los2020isogeometric,houston2020eliminating,calo2020adaptive}.
Most of these novel methods can be derived from Demkowicz and Gopalakrishnan's discontinuous Petrov--Galerkin (DPG) method with optimal test functions \cite{demkowicz2010class,demkowicz2011class}, however, many of these methods' key features can be traced at least to the first-order system least squares (FOSLS) methods introduced by Cai \emph{et al.} in \cite{cai1994first,cai1997first}.
Each of the methods mentioned above arise from a minimum residual energy principle.
Consequently, they are numerically stable and tend to have a built-in local error estimator which facilitates adaptive mesh refinement \cite{bochev2009least,demkowicz2012class,carstensen2014posteriori,Keith2017Discrete}.
It is common to refer to these non-standard finite element methods as \emph{minimum residual methods}.

A closely related class of methods derive from a very different category of energy principles which come about by minimizing the norm of all possible solutions under a variational equation constraint.
In this paper, we refer to this alternative class of methods as \emph{minimum norm methods}.
The most notable example are the $\mcL\mcL^\ast$ methods (also known as FOSLL* methods), first introduced by Cai \emph{et al.} in \cite{cai2001first}.

Minimum norm methods arise naturally when the adjoint formulation of a minimum residual method is considered \cite{Keith2018thesis}.
This may come about during the mathematical analysis of a minimum residual method \cite{demkowicz2013primal,fuhrer2017superconvergence} or when using a minimum residual method in an optimization setting; for instance, when designing error estimators based on an extrinsic quantity of interest \cite{Keith2017Goal,valseth2020goal,rojas2020residual,chakraborty2020optimal}.
Such methods may also be derived entirely on their own \cite{cai2001first,manteuffel2005first,lee2007fosll,cai2015div,lee2015fosll,brugiapaglia2015compressed,Keith2017DPGstarICESReport,brugiapaglia2018theoretical,demkowicz2018dpg,brugiapaglia2020compressive}.
In every setting, the question of accuracy is critical.

The reason for this work is to clear up an apparent confusion in the literature about the numerical accuracy of minimum norm methods.
Certain authors have pointed out that their particular minimum norm method has poor accuracy at high polynomial orders \cite{demkowicz2018dpg,valseth2020goal}.
However, given how well-established $\mcL\mcL^\ast$ methods are in the literature, other scientists have expressed an instinctive skepticism toward these assertions.
The aim of this paper is to settle the dispute by demonstrating that even $\mcL\mcL^\ast$ methods have high-order accuracy limitations in conventional test problems.

The need for this contribution is imminent as the number of new minimum norm methods appearing in the literature seems to only be accelerating.
In addition, within the part community which already believes in the accuracy limitation, there is some disagreement on its exact cause \cite{demkowicz2018dpg,valseth2020goal}.

This work does not aim to undermine certain important advantages that minimum norm methods also hold.
For instance, we do not dispute that the structure of minimum norm methods inherits numerical stability.
This means that these methods can be applied more or less ``as is'' to wide variety of problems; \emph{e.g.}, \cite{cai2001first,manteuffel2005first,lee2007fosll,cai2015div,lee2015fosll}.
They also typically lead to symmetric positive definite stiffness matrices, which obviously endears them to the most efficient linear solvers.

To make this work more accessible, we have chosen to focus our analysis on the $\mcL\mcL^\ast$ method (in its most common form, also known as the FOSLL$^\ast$ method) \cite{cai2001first,manteuffel2005first,lee2007fosll,cai2015div,lee2015fosll}, but our main results are generalizable to the entire category of minimum norm methods.
Our choice alleviates the need to introduce extensive notation which would be necessary to cover all minimum norm methods at once; \emph{e.g.}, \cite{Keith2018thesis,demkowicz2018dpg}.
It is also desirable because $\mcL\mcL^\ast$ is likely the oldest minimum norm method, the easiest to understand, and, certainly, the most well-known.
It is the author's opinion that the analysis in this setting is sufficiently direct and sufficiently generalizable, that any interested reader would be able to apply it to the setting and notation of their own method.
We now introduce the primary notation and restate the goal of this paper in a more mathematically precise way.

\subsection{Notation and intention} \label{sub:notation}

In what follows, the $\mcL\mcL^\ast$ method is a finite element method which derives from the following PDE-constrained norm minimization problem on a domain $\Omega\subset\R^n$:
\begin{equation}
\label{eq:NormMinProb}
  \min_{u\in L^2(\Omega)}
  \frac{1}{2}
  \| u\|_{L^2(\Omega)}^2
  \quad
  \text{subject to}
  \quad
    \int_\Omega u\cdot \mcL^\ast v \dd\!x  = \int_\Omega f\cdot v \dd\!x \quad \text{for all } v\in \mathrm{Dom}(\mcL^\ast).
    \end{equation}
\revised{
Here, $\mcL^\ast \colon L^2(\Omega) \supset \mathrm{Dom}(\mcL^\ast) \to L^2(\Omega)$ is a closed linear operator (with its domain $\mathrm{Dom}(\mcL^\ast)$ dense in $L^2(\Omega)$) and $f\in L^2(\Omega)$ is some prescribed function.
}

Let $u\in H^{s}(\Omega)$ and $u_{hp} \in V_{hp}$ be its order-$p$ finite element approximation on a shape-regular mesh $\mesh$.
We say that the corresponding finite element method delivers limited convergence rates if there exists a maximum exponent $r \geq 0$, \emph{independent} of $p$ and $s$, such that
the following upper bound is sharp for all sufficiently large $p$ and $s$,
\begin{equation}
\label{eq:RateLimitedDefn}
      \|u-u_{hp}\|_{L^2} \leq C h^{r}\|u\|_{H^s}\,.
  \end{equation}
Here, the constant $C$ depends at most on $p,~s,$ and $\Omega$.
Meanwhile, $\|\cdot\|_{L^2} = \|\cdot\|_{H^0}$ and $\|\cdot\|_{H^s}$, $s>0$, denote $L^2(\Omega)$ and $H^s(\Omega)$ norms, respectively.
Throughout this work, we use the same notation above for the norms of vector- and tensor-valued functions.

In this work, we show that the standard $\mcL\mcL^\ast$ finite element method, like other minimum norm methods, delivers limited convergence rates.

\subsection{Layout} \label{sub:layout}

In \Cref{sec:derivation} we give a formal derivation of the $\mcL\mcL^\ast$ method.
In \Cref{sec:model_problem} we introduce a general second-order elliptic model problem.
In \Cref{sec:a_priori_error_estimation} we perform an \emph{a priori} error estimation for the $\mcL\mcL^\ast$ method applied to the model problem.
In \Cref{sec:numerical_evidence}, we give numerical evidence to verify the theory put forward in \Cref{sec:a_priori_error_estimation}.
Finally, in \Cref{sec:discussion}, we give a short summary of our findings and various concluding remarks.

\section{Derivation and mixed formulation} \label{sec:derivation}

One may arrive at the $\mcL\mcL^\ast$ method in a variety of ways.
Most notably, one can begin with a minimum residual principle with a problem-dependent negative norm and then derive the variational equations which uniquely characterize its optimizer \cite{cai2001first,bochev2009least}.
This is the most well-established approach, however, it requires the definition of various symbols that are not required in this work.
Therefore, we opt for a short derivation which allows us to proceed immediately to an explicit example and its analysis.
The interested reader is referred to \cite{bochev2009least} and references therein for a thorough motivation of the $\mcL\mcL^\ast$ method.

\revised{
We begin with a linear boundary value problem (BVP) over a domain $\Omega\subset\R^n$:
\begin{equation}
  \mcL u = f,
\label{eq:GeneralPDE}
\end{equation}
where $\mcL$ is a linear differential operator and $f\in L^2(\Omega)$ is some prescribed function.
More specifically, we assume that $\mcL \colon L^2(\Omega) \supset \mathrm{Dom}(\mcL) \to L^2(\Omega)$ is a closed, bijective linear operator and that its domain, written $\mathrm{Dom}(\mcL)$, is dense in $L^2(\Omega)$ and equipped with the graph norm $\|u\|_{H(\mcL)}^2 =\|u\|_{L^2}^2 + \|\mcL u \|_{L^2}^2$.
If $\mcL$ is bounded below, meaning
\begin{equation}
\label{eq:BoundedBelowPrimal}
  \|v\|_{L^2} \leq C\|\mcL v\|_{L^2}
  \quad
  \text{for all }
  v \in \mathrm{Dom}(\mcL)
  ,
\end{equation}
then, by the Closed Range Theorem (for closed operators), its adjoint is also bounded below \cite[Section~4.3]{DemkowiczClosedRange},
\begin{equation}
\label{eq:BoundedBelow}
  \|v\|_{L^2} \leq C\|\mcL^\ast v\|_{L^2}
  \quad
  \text{for all }
  v \in \mathrm{Dom}(\mcL^\ast)
  .
\end{equation}
By~\cref{eq:BoundedBelowPrimal}, the \emph{strong} variational formulation,
\begin{equation+}
  \text{Find }u\in \mathrm{Dom}(\mcL)
  \text{ satisfying}
  \quad
  (\mcL u, v) = (f,v)
  \quad
  \text{for all } v\in L^2(\Omega),
\end{equation+}
is well-posed.
Likewise, by~\cref{eq:BoundedBelow}, the \emph{ultraweak} variational formulation,
\begin{equation}
  \text{Find }u\in L^2(\Omega)
  \text{ satisfying}
  \quad
  (u, \mcL^\ast v) = (f,v)
  \quad
  \text{for all } v\in \mathrm{Dom}(\mcL^\ast),
\label{eq:UltraweakForm}
\end{equation}
is also well-posed (with respect to the graph norm $\|v\|_{H(\mcL^\ast)}^2 = \|v\|_{L^2}^2 + \|\mcL^\ast v \|_{L^2}^2$) \cite{DemkowiczClosedRange}.
}

\revised{
If we define the auxiliary variable $u = \mcL^\ast \lambda$, then~\cref{eq:UltraweakForm} can be rewritten as
\begin{subequations}
\label{eq:LL*primal}
\begin{equation}
\label{eq:LL*primalformulation}
  \text{Find }\lambda\in \mathrm{Dom}(\mcL^\ast)
  \text{ satisfying}
  \quad
  (\mcL^\ast \lambda, \mcL^\ast v) = (f, v)
  \quad\text{for all }v \in \mathrm{Dom}(\mcL^\ast).
\end{equation}
This variational formulation is the foundation of the $\mcL\mcL^\ast$ finite element method.
Indeed, introducing the finite-dimensional subspace $V_{hp}\subsetneq \mathrm{Dom}(\mcL^\ast)$, the $\mcL\mcL^\ast$ solution is defined as $u_{hp} = \mcL^\ast\lambda_{hp}\in \mcL^\ast(V_{hp})$, where $\lambda_{hp}$ is the unique solution to the following discrete variational problem:
\begin{equation}
\label{eq:LL*primalmethod}
  \text{Find }\lambda_{hp}\in V_{hp}
  \text{ satisfying}
  \quad
  (\mcL^\ast \lambda_{hp}, \mcL^\ast v) = (f, v)
  \quad\text{for all }v \in V_{hp}.
\end{equation} 
\end{subequations}
}

In the coming arguments, it is helpful to also consider the mixed formulation of~\cref{eq:LL*primalformulation}:
\begin{subequations}
\label{eq:mixedmethods}
\begin{equation}
  \begin{gathered}
    \text{Find }(u,\lambda)\in L^2(\Omega)\times \mathrm{Dom}(\mcL^\ast)
    \text{ satisfying}
    \hspace{60mm}
    \\
    \hspace{50mm}
    \left\{
    \begin{alignedat}{3}
      (u,w) - (\mcL^\ast \lambda, w) &= 0
      \quad&&\text{for all }w \in L^2(\Omega),\\
      (u, \mcL^\ast v) &= (f, v)
      \quad&&\text{for all }v \in \mathrm{Dom}(\mcL^\ast).
    \end{alignedat}
    \right.
  \end{gathered}
\label{eq:LL*mixedformulation}
\end{equation}
Likewise, the related discrete formulation (cf.~\cref{eq:LL*primalmethod}) is as follows:
\begin{equation}
  \begin{gathered}
    \text{Find }(u_{hp},\lambda_{hp})\in \mcL^\ast(V_{hp})\times V_{hp}
    \text{ satisfying}
    \hspace{50mm}
    \\
    \hspace{50mm}
    \left\{
    \begin{alignedat}{3}
      (u_{hp},w) - (\mcL^\ast \lambda_{hp}, w) &= 0
      \quad&&\text{for all }w \in \mcL^\ast(V_{hp}),\phantom{L^2}\!\!\\
      (u_{hp}, \mcL^\ast v) &= (f, v)
      \quad&&\text{for all }v \in V_{hp}.
    \end{alignedat}
    \right.
  \end{gathered}
\label{eq:LL*mixedmethod}
\end{equation}
\end{subequations}
It is important to emphasize that \cref{eq:LL*mixedformulation,eq:LL*mixedmethod} are equivalent to~\cref{eq:LL*primalformulation,eq:LL*primalmethod}, respectively; they are simply written differently.
This mixed method perspective is very helpful.
For instance, from~\cref{eq:LL*mixedformulation} it is evident that $\lambda$ is the Lagrange multiplier associated to the minimization problem~\cref{eq:NormMinProb}.

\section{Model problem} \label{sec:model_problem}

In this section, we consider the second-order linear homogeneous boundary value problem (BVP)
\begin{align}
\label{eq:PDE}
  -\div(\bmA \grad u) + \bmb\cdot\grad u + c\, u
  &=
  f
  \quad
    \text{in } \Omega,
\label{eq:BC}
            \end{align}
with the boundary condition $u = 0$ on $\partial\Omega$.
Here, $\Omega\subset \R^d$, $d=2,3$, is a polyhedral domain, $\bmA^{-1} \in [L^\infty(\Omega)]^{d\times d}$ is a symmetric positive definite matrix-valued function almost everywhere, $\bmb \in [L^\infty(\Omega)]^d$ is almost everywhere divergence free, $c \in L^\infty(\Omega)$ is almost everywhere non-negative, and $f\in L^2(\Omega)$.
It is well-known that the solution of this problem, $u$, belongs to $H^1_0(\Omega)$.

We follow \cite{cai2015div} and write~\cref{eq:PDE} as the following first-order system of equations involving a new solution variable $\bsigma \in H(\div,\Omega) = \{\bsigma \in [L^2(\Omega)]^d \colon \div(\bsigma) \in L^2(\Omega)\}$:
\begin{equation}
      \begin{aligned}
      \bmA^{-1}\bsigma + \grad u
      &=
      \bm{0}
      \quad
      &&\text{in } \Omega,
      \\
      \div\bsigma - \bmb\cdot \bmA^{-1}\bsigma + c\, u
      &=
      f
      \quad
      &&\text{in } \Omega.
    \end{aligned}
  \end{equation}
\revised{
An alternative way of expressing these equations is to write $\LL \bmu = \bmf$, where the operator $\mcL\colon [L^2(\Omega)]^d\times L^2(\Omega) \supset H(\div,\Omega)\times H^1_0(\Omega)\to [L^2(\Omega)]^d\times L^2(\Omega)$ is defined
\begin{equation}
\label{eq:ModelOperator}
          \LL \bmu = ( \bmA^{-1}\bsigma + \grad u,\, \div\bsigma - \bmb\cdot \bmA^{-1}\bsigma + c\, u)
  ,
  \quad
  \bmu = (\bsigma, u)
    ,
\end{equation}
and the load $\bmf \in [L^2(\Omega)]^{d}\times L^2(\Omega)$ is defined
\begin{equation}
\label{eq:ModelLoad}
  \bmf = (\bm0,f)
  .
\end{equation}
A straightforward calculation (see, \emph{e.g.}, \cite[Section~2.3]{cai2015div}) can be performed to show that the $L^2$-adjoint of the operator~\cref{eq:ModelOperator}, $\mcL^\ast\colon [L^2(\Omega)]^d\times L^2(\Omega) \supset H(\div,\Omega)\times H^1_0(\Omega)\to [L^2(\Omega)]^d\times L^2(\Omega)$, is simply
\begin{equation}
\label{eq:ModelAdjoint}
  \LL^* \bmv  = ( \bmA^{-1}\btau - \bmA^{-1}\bmb\, v - \grad v,\, c\,v - \div \btau)
  ,
  \quad
  \bmv = (\btau, v)
    .
\end{equation}
}

\revised{
In the notation of \Cref{sec:derivation}, we see that $\mathrm{Dom}(\mcL) = \mathrm{Dom}(\mcL^\ast) = H(\div,\Omega)\times H^1_0(\Omega)$.
Accordingly, the $\mcL\mcL^\ast$ formulation of~\cref{eq:PDE} is
\begin{equation}
\label{eq:LL*Formulation}
  \begin{gathered}
  \text{Find } \blambda = (\bzeta,\lambda) \in \mathrm{Dom}(\mcL^\ast)
  \text{ satisfying}
  ~~
  ( \LL^* \blambda, \LL^* \bmv ) = ( \bmf, \bmv )
  \quad\text{for all } \bmv \in \mathrm{Dom}(\mcL^\ast)
  ,
                      \end{gathered}
\end{equation}
with the definitions~\cref{eq:ModelLoad,eq:ModelAdjoint} given above.
From the solution of this variational problem, one then reconstructs the solution $\bmu = \LL^*\blambda$.
}

\subsection{Properties of the Lagrange multiplier} \label{sub:properties_of_the_lagrange_multiplier}

It is instructive to write out the strong form of~\cref{eq:LL*Formulation} for the Laplace operator; i.e., let $\bmA = \bmI$ (the identity matrix), and let both $\bmb$ and $c$ be zero.
Doing so will lend insight to the regularity required of the Lagrange multiplier $\blambda = (\bzeta,\lambda)$ and the intricate relationship between it and $\bmu = (\bsigma,u)$ and $f$.
We therefore turn to the simplified variational equation
\begin{equation}
  (\bzeta - \grad \lambda, \btau - \grad v) + (\div \bzeta, \div \btau)
  = ( f, v )
  \quad\text{for all }(\btau,v) \in \mathrm{Dom}(\mcL^\ast)
  ,
\label{eq:PoissonVE}
\end{equation}
and the expressions $\bsigma = \bzeta - \grad\lambda$ and $u = -\div\bzeta$.

After integrating~\cref{eq:PoissonVE} by parts, one arrives at the following coupled system of second-order equations for the variables $\bzeta,~\lambda$:
\begin{subequations}
\label{eq:LLstarSystem}
\begin{align+}
\label{eq:LLstarSystemA}
  \bzeta - \grad\lambda - \grad\div\bzeta &= \bm{0}\quad \text{in } \Omega,\\
\label{eq:LLstarSystemB}
  \div\bzeta - \Delta\lambda &= f\quad \text{in } \Omega,    \end{align+}
\end{subequations}
with the boundary conditions $\lambda = \div\bzeta = 0$ on $\partial\Omega$.
The reader may now notice, after simply substituting $\div\bzeta = -u$ into~\cref{eq:LLstarSystemB}, that $\lambda$ is the unique solution of the following Dirichlet problem:
\begin{equation}
  -\Delta \lambda = f + u
  \quad 
  \text{in } \Omega,
  \quad
  \text{with }
  \lambda = 0
  \quad
  \text{on } \partial\Omega
  .
\end{equation}
The regularity of solutions to this problem have been studied extensively \cite{Grisvard,MR2597943,grisvard2011elliptic}.
One well-known result is that if $\Omega$ is convex or globally $C^2$, then $\lambda \in H^2(\Omega)\cap H^1_0(\Omega)$.
In two dimensions, this result can be improved if one invokes information about the angle of the largest corner in the domain, $\theta < \pi$.
In general, one finds there is a constant $C>0$, depending on $s$, such that
\begin{equation}
\label{eq:EllipticRegularity}
  \|\lambda\|_{H^{s+2}} \leq C \|f + u\|_{H^{s}}
  ,
  \quad
  0 \leq s < s_0
  ,
\end{equation}
where $s_0 = \min\{1,\pi/\theta - 1\}$; see, \emph{e.g.}, \cite{bacuta2003regularity2}.
The regularity of $\lambda$ deteriorates in another well-understood way in the presence of domains with re-entrant corners; see, \emph{e.g.}, \cite{bacuta2003regularity,kondrat1967boundary}.
More limited bounds also hold with non-homogeneous boundary conditions \cite{bacuta2003regularity2,bacuta2003regularity} or when $\bmA \neq \bmI$ and when $\bmb$ and $c$ are non-zero; see, \emph{e.g.}, \cite{Grisvard,grisvard2011elliptic} and \cite[Chapter~6.3]{MR2597943}.

  At this point, it is still not clear why the regularity limitation of the Lagrange multiplier component $\lambda$ should affect the convergence rate of the $\mcL\mcL^*$ method or, in fact, any minimum norm method.
  Indeed, based solely on the discussion above, it appears only that both $\lambda$ and $u$ have similar regularity properties.
  This is an important observation in its own right since various authors have stated that $\lambda$ should have higher regularity than $u$, simply because it derives from the higher-order PDE~\cref{eq:LLstarSystem}; see, \emph{e.g.}, \cite[p.~483]{bochev2009least}.
  In addition, it is known that the residual function (sometimes called the ``error representation function'' \cite{demkowicz2015encyclopedia}), a different secondary variable which arises naturally in minimum residual methods, is similarly restricted; see, \emph{e.g.}, \cite{Keith2018thesis} and references therein.
  The key ingredient that we are missing appears in the stability bound which arises from the saddle-point analysis.

\subsection{Discretization} \label{sub:discretization}

In \Cref{sec:a_priori_error_estimation}, we use~\cref{eq:EllipticRegularity} to argue that the specific form of the saddle-point problems~\cref{eq:LL*mixedformulation,eq:LL*mixedmethod} limits the convergence rates of $\mcL\mcL^\ast$ methods for certain second-order elliptic PDEs.
In addition, we explain how the same analysis may be replicated and applied to other minimum norm methods more generally.
First, however, we must introduce a specific discretization we aim to analyze.
For continuity, we choose to follow the same discretization of~\cref{eq:LL*Formulation} considered in \cite{cai2015div}.

Let $\mesh$ be a shape-regular triangulation of $\Omega$ into simplices.
For each element $K\in\mesh$, we denote the space of polynomials of degree less than or equal to $p$ on $K$ by $P_p(K)$.
Accordingly, the local Raviart--Thomas space of order $p$ on $K$ \cite{raviart1977mixed} is defined
\begin{equation}
  \mathcal{RT}_p(K)
  =
  P_p(K)^d + \bfx P_p(K)
  ,
\end{equation}
where $\bfx = (x_1,\ldots,x_d)$.

We may now define the approximation space $V_{hp}$ using a globally $H(\div)$-conforming Raviart--Thomas space and $H^1$-conforming (i.e., continuous) piecewise polynomials of degree at most $p+1$.
That is, $\bmV_{hp} = \Sigma_p(\mesh) \times V_{p+1}(\mesh)$, where
\begin{align}
  \Sigma_p(\mesh) &= \{\btau \in H(\div,\Omega) \colon \btau|_K \in \mathcal{RT}_p(K) ~ \forall K\in \mesh\},
  \\
  V_{p+1}(\mesh) &= \{v \in H^1_0(\Omega) \colon v|_K \in P_{p+1}(K) ~ \forall K\in \mesh\}.
\end{align}
With these definitions in hand, we arrive at the following discrete variational formulation:
\begin{equation}
\label{eq:LL*primalmethodElliptic}
  \text{Find }\blambda_{hp} = (\bzeta_{hp},\lambda_{hp})\in \bmV_{hp}
  \text{ satisfying}
  ~~
  (\mcL^\ast \blambda_{hp}, \mcL^\ast \bmv) = (\bmf, \bmv)
  \quad\text{for all }\bmv \in \bmV_{hp}.
\end{equation}
After solving this discrete variational problem, one may reconstruct the discrete solution $\bmu_{hp} = (\bsigma_{hp}, u_{hp}) = \mcL^\ast\blambda_{hp}$.

Certain approximation properties of the spaces $\Sigma_p(\mesh)$ and $V_{p+1}(\mesh)$ will be important in the next section.
Let $\Pi_{\grad}^{p+1} \colon H^1_0(\Omega) \to V_{p+1}(\mesh)$ denote the Scott--Zhang projection operator \cite{scott1990finite}, which has the property
\begin{equation}
\label{eq:SZprojector}
  \| u - \Pi_{\grad}^{p+1} u\|_{H^1}
  \leq
  Ch^{r} \| u \|_{H^{r+1}}
  ,
  \quad
    0 \leq r \leq p+1
  .
\end{equation}
Moreover, let $\Pi_{\div}^{p} \colon H(\div,\Omega) \cap H^1(\Omega) \to \Sigma_p(\mesh)$ denote the Raviart--Thomas operator, which satisfies \cite[Proposition~2.5.4]{Boffi2013}
\begin{subequations}
\begin{gather}
\label{eq:RTprojectorL2}
  \| \bsigma - \Pi_{\div}^{p} \bsigma\|_{L^2}
  \leq
  Ch^{r} | \bsigma |_{H^{r}}
  ,
  \quad
    1 \leq r \leq p+1
  ,
  \\
\label{eq:RTprojector}
  \| \div(\bsigma - \Pi_{\div}^{p} \bsigma) \|_{L^2}
  \leq
  Ch^{r} | \div \bsigma |_{H^{r}}
      ,
  \quad
    1 \leq r \leq p+1
  .
\end{gather}
\end{subequations}

\section{A priori error estimation} \label{sec:a_priori_error_estimation}

If the variational problems~\cref{eq:LL*primalformulation,eq:LL*primalmethod} are well-posed, then so are~\cref{eq:LL*mixedformulation,eq:LL*mixedmethod}.
Therefore, in the analysis of any well-posed $\mcL\mcL^\ast$ method, we may invoke the following a priori error estimate from the standard theory of mixed methods \cite{Boffi2013}:
\begin{theorem}
  \label{thm:apriori}
    There is a constant $C$ such that the complete $\mcL\mcL^\ast$ solution $(u_{hp},\lambda_{hp})$ satisfies the error
      estimate
  \[
    \| u - u_{hp} \|_{L^2} + \| \lambda - \lambda_{hp} \|_{H(\mcL^\ast)}
    \le 
    C 
    \bigg[
    \inf_{\mu \in \mcL^\ast(V_{hp})} 
    \| u - \mu \|_{L^2} +
    \inf_{\nu \in V_{hp}} \| \lambda - \nu \|_{H(\mcL^\ast)}
    \bigg].
  \]  
\end{theorem}
From that point on, a Bramble--Hilbert argument will deliver the following inequality:
  \begin{equation}
  \label{eq:LLapriori}
    \| u - u_{hp} \|_{L^2} + \| \lambda - \lambda_{hp} \|_{H(\mcL^\ast)}
    \le 
    C 
    h^{r_0} \big(\|u\|_{H^{r_1}} + \|\lambda\|_{H^{r_2}}\big)
    ,
  \end{equation}
where the constants $r_0,r_1,r_2>0$ depend on the functional setting and discretization used.
A similar sequence of arguments could be made for any minimum norm method.
For example, one may refer to \cite[Section~3.2]{demkowicz2018dpg} which presents exactly this form of result for an ultraweak DPG* method.

As our focus is on the $\mcL\mcL^\ast$ method, a number of corollaries to~\cref{thm:apriori} are in order.
Here, and from now on, we write $A \eqsim B$ if there exist constants $D_1, D_2 > 0$, independent of both the mesh and the polynomial order, such that $A \leq D_1 B$ and $A \geq D_2 B$.
First, since $u = \mcL^\ast\lambda$ and $u_{hp} = \mcL^\ast\lambda_{hp}$, it follows that
\begin{equation}
\label{eq:Equivalence}
  \|u-u_{hp}\|_{L^2}
  =
  \|\mcL^\ast\lambda - \mcL^\ast\lambda_{hp}\|_{L^2}
  \eqsim
  \|\lambda - \lambda_{hp}\|_{H(\mcL^\ast)}
  ,
\end{equation}
with the latter equivalence following from the boundedness below of $\mcL^\ast$; cf.~\cref{eq:BoundedBelow}.
Thus, we see that~\cref{thm:apriori} is equivalent to
\begin{equation}
\label{eq:Corollary}
  \|u-u_{hp}\|_{L^2}
  \eqsim
  \|\lambda - \lambda_{hp}\|_{H(\mcL^\ast)}
  \leq
  C \inf_{\nu \in V_{hp}}
  \| \lambda - \nu \|_{H(\mcL^\ast)}
  .
\end{equation}
At this point, it is clear that the convergence rate of the $\mcL\mcL^\ast$ method depends solely on the regularity of $\lambda$.
Exploiting~\cref{eq:Equivalence} differently, one may also write
\begin{equation}
  \|u-u_{hp}\|_{L^2}
  \leq
  C
  \inf_{\mu \in \mcL^\ast(V_{hp})} 
  \| u - \mu \|_{L^2}
  .
\end{equation}
Through~\cref{eq:Equivalence}, this inequality is equivalent to~\cref{eq:Corollary}.
Thus, we have the simple interpretation that the discrete space $\mcL^\ast(V_{hp})$ is typically just a poor choice for approximating $u$.

For completeness, we prove a version of~\cref{eq:LLapriori} which holds for the specific $\mcL\mcL^\ast$ method introduced in \Cref{sec:model_problem}.
Our result is related to \cite[Theorem~4.1]{cai2015div}, which also deals with the same $\mcL\mcL^\ast$ method, but involves an additional regularity assumption on the Lagrange multiplier variables.
After the proof, we explain why our result shows that $\mcL\mcL^\ast$ methods are rate limited.
We also explain why this convergence behavior is to be expected from any standard minimum norm method.

\begin{proposition}
\label{prop:LLstarCorollary}
      There is a constant $C$ such that the complete $\mcL\mcL^\ast$ solution $(\bmu_{hp},\blambda_{hp})$ satisfies the error estimate
                    \begin{equation}
  \label{eq:LLstarThm}
    \begin{gathered}
      \| \bsigma - \bsigma_{hp}\|_{L^2} + \| u - u_{hp} \|_{L^2} + \| \bzeta -\bzeta_{hp} \|_{H(\div)} + \| \lambda - \lambda_{hp} \|_{H^1}
            \hspace{25mm}
      \\
      \hspace{60mm}
      \le 
      C 
            h^r \big(|\bzeta|_{H^r} + |\div\bzeta|_{H^r} + \|\lambda\|_{H^{r+1}}\big)
                  \,,
    \end{gathered}
  \end{equation}
  for all $1 \leq r \leq p+1$.
\end{proposition}
\begin{proof}
  It follows from \cite[Theorem~2.2]{cai2015div} that $\mcL^\ast\colon H(\div,\Omega)\times H^1(\Omega)\to [L^2(\Omega)]^d\times L^2(\Omega)$, defined in~\cref{eq:ModelAdjoint}, is continuous and bounded from below~\cref{eq:BoundedBelow}.
  Therefore, we begin with~\cref{eq:Corollary}, rewritten as follows:
  \begin{equation}
    \begin{gathered}
      \| (\bsigma - \bsigma_{hp},u - u_{hp}) \|_{L^2} + \| (\bzeta -\bzeta_{hp},\lambda - \lambda_{hp}) \|_{H(\mcL^\ast)}
                        \le 
      C 
            \inf_{(\bnu,\nu) \in \bmV_{hp}}
      \| ( \bzeta - \bnu, \lambda - \nu )\|_{H(\mcL^\ast)}
      .
    \end{gathered}
  \end{equation}
  A straightforward computation shows that $\| (\bnu,\nu) \|_{H(\mcL^\ast)} \eqsim \| (\bnu,\nu)\|_{H(\div)\times H^1}$ for all $(\bnu,\nu) \in H(\mcL^\ast)$.
  Hence, by~\cref{eq:SZprojector,eq:RTprojectorL2,eq:RTprojector}, we have
  \begin{align}
    \inf_{(\bnu,\nu) \in \bmV_{hp}}
    \| ( \bzeta - \bnu, \lambda - \nu )\|_{H(\mcL^\ast)}^2
    &\eqsim
    \inf_{(\bnu,\nu) \in \bmV_{hp}}
    \| ( \bzeta - \bnu, \lambda - \nu )\|_{H(\div)\times H^1}^2
    \\
    &=
    \inf_{\bnu \in \Sigma_{p}(\mesh)} \| \bzeta - \bnu\|_{H(\div)}^2
    +
    \inf_{\nu \in V_{p+1}(\mesh)} \| \lambda - \nu\|_{H^1}^2
    \\
    &\leq
    \| \bzeta - \Pi^{p}_{\div}\bzeta \|_{H(\div)}^2
    +
    \| \lambda - \Pi^{p+1}_{\grad}\lambda \|_{H^1}^2
    \\
    &\leq
    Ch^{2r}\big(| \bzeta |_{H^{r}}^2 + | \div\bzeta |_{H^{r}}^2 + \| \lambda \|_{H^{r+1}}^2\big)
    ,
  \end{align}
  which completes the proof.
\end{proof}

\Cref{prop:LLstarCorollary} shows that the convergence rate of the $\mcL\mcL^\ast$ method is dependent on the regularity of the Lagrange multiplier $\blambda = (\bzeta,\lambda)$.
For $\bmA \in [W^{r,\infty}(\Omega)]^{d\times d}$, it is straightforward to show that
\begin{equation}
\label{eq:PrimalVariableBound}
  \|(\bsigma,u)\|_{H^{r}} \leq C \|u\|_{H^{r+1}}
  ,
\end{equation}
for all $r \geq 0$.
Thus, the regularity of the solution variable $\bsigma$ is controlled by the regularity of $u$.
This inequality naturally begs the question whether the regularity of the Lagrange multiplier $\blambda$ can also be controlled by the regularity of $u$.
Such assumptions can be found, for example, in \cite{cai2001first,cai2015div}.
For additional motivation, if one assumes that there is a mesh-independent constant $C>0$ such that
\begin{equation}
\label{eq:StrongRegularity}
  |\bzeta|_{H^r} + |\div\bzeta|_{H^r} + \|\lambda\|_{H^{r+1}}
    \leq
  C \|(\bsigma,u)\|_{H^{r}}
      ,
\end{equation}
for all $r \geq 0$ (cf. \cite[Theorem~4.1]{cai2015div}), then the convergence rate determined by~\cref{eq:LLstarThm} would indeed depend only on the regularity of $u$.

Unfortunately, as was shown in \Cref{sub:properties_of_the_lagrange_multiplier}, the regularity of the Lagrange multiplier \emph{cannot} be determined solely based on the regularity of the solution.
In fact, the regularity of $\lambda$ is \emph{independent} of the regularity of $u$.
For instance, $u$ may be infinitely smooth, but $\lambda \in H^{s+2}(\Omega)$, only up to some finite $s<s_0$.
The next section is devoted to two examples which demonstrate this; cf. \Cref{sub:square_domain_dirichlet}.
In the first example, both $u$ and $\lambda$ are infinitely smooth, meanwhile, in second example, $u$ is a polynomial but the Lagrange multiplier $\lambda\notin H^{s}(\Omega)$, for any $s\geq 3$.
In the second example, we get only the best $h$-uniform convergence rate predicted by the following corollary of~\Cref{prop:LLstarCorollary}; namely, first-order convergence when $p=0$, but only second-order convergence for all $p\geq 1$.

\begin{corollary}
\label{cor:RateLimited}
  The $\mcL\mcL^\ast$ method for the Poisson equation on a convex polygonal domain $\Omega\subset\R^2$, delivers limited convergence rates.
  In general, it holds that
  \begin{equation}
    \| u - u_{hp} \|_{L^2}
    +
    \| \bsigma - \bsigma_{hp} \|_{L^2}
    \leq
    C 
    h^{\min\{r,p+1,s_0+1-\epsilon\}} \|u\|_{H^{r+1}}
    ,
  \end{equation}
  for all $r \geq 1$ and $\epsilon>0$, where $s_0 = \min\{1,\pi/\theta - 1\}$ and $\theta < \pi$ is the angle of the largest corner in the domain $\Omega$.
\end{corollary}
\begin{proof}
  Recall the identities $-\div \bzeta = u$ and $\bzeta = \grad\lambda - \grad u$ from \Cref{sub:properties_of_the_lagrange_multiplier}.
  Thus, $\| \bzeta \|_{H^{r}} + \| \div\bzeta \|_{H^{r}} \leq C ( \| \lambda \|_{H^{r+1}} + \| u \|_{H^{r+1}} )$ and, by~\cref{eq:LLstarThm},
  \begin{equation}
    \| u - u_{hp} \|_{L^2}
    +
    \| \bsigma - \bsigma_{hp} \|_{L^2}
    \leq
    C 
    h^r \big(\|u\|_{H^{r+1}} + \|\lambda\|_{H^{r+1}}\big)
    ,
  \end{equation}
  for all $1 \leq r < p+1$.
  Next, recall~\cref{eq:EllipticRegularity}, which implies that
  \begin{equation}
    \|\lambda\|_{H^{s+2}}
    \leq
    C \|\Delta u\|_{H^{s}} + \|u\|_{H^{s}} 
    \leq
    C \|u\|_{H^{s+2}} 
    ,
  \end{equation}
  for all $0 \leq s < s_0$.
  Combining both these bounds, and the assumptions which deliver them, we arrive at the required identity,
  \begin{align}
      \| u - u_{hp} \|_{L^2}
      +
      \| \bsigma - \bsigma_{hp} \|_{L^2}
      &\leq
      C 
      h^{\min\{r,p+1,s_0+1-\epsilon\}} \|u\|_{H^{r+1}}
      ,
  \end{align}
  even when $u\in C^\infty(\overline{\Omega})$.
\end{proof}

\begin{remark}
For other minimum norm methods, the upshot of~\Cref{cor:RateLimited} is similar \cite{demkowicz2018dpg,Keith2018thesis}.
Indeed, due to the unique characteristics of minimum norm methods, one will eventually uncover an inequality like~\cref{eq:LLapriori} and, generally, the regularity of $\lambda$ will not be controlled entirely by the regularity of $u$.
Thus, even if $u$ is infinitely smooth, one will not be able to guarantee arbitrary convergence rates.
\end{remark}

\section{Numerical examples} \label{sec:numerical_evidence}

In this section, we numerically verify that the $\mcL\mcL^\ast$ method delivers limited uniform $h$-convergence rates, as defined in~\Cref{sub:notation}.
To this end, we note that it is sufficient that we make the same simplifications as in \Cref{sub:properties_of_the_lagrange_multiplier}.
To conduct our experiments, we used the finite element software FEniCS \cite{alnaes2015fenics}.

\subsection{Set-up} \label{sub:set_up}
We again consider the model problem~\cref{eq:PDE}, but this time with the possibility of a non-homogeneous Dirichlet boundary condition.
To this end, let $u_0 \in H^{\onehalf}(\bdry\Omega)$ \revised{and $f\in L^2(\Omega)$}.
After making the simplifying assumptions $\bmA = \bmI$, $\bmb = \bm{0}$, and $c = 0$, we arrive at the elliptic BVP
\begin{equation}
  -\Delta u = f
  \quad
  \text{in } \Omega,
  \quad
  \text{with }
  u = u_0
  \quad
  \text{on } \partial\Omega.
                  \label{eq:ModelProblem}
\end{equation}
In this setting (with the non-homogeneous boundary condition $u|_{\partial\Omega} = u_0$), a straightforward computation still shows that
\begin{equation}
  -\Delta \lambda = f+u
  \quad
  \text{in } \Omega,
  \quad
  \text{with }
  \lambda = 0
  \quad
  \text{on } \partial\Omega.
\label{eq:LambdaPoissonInhomogeneous}
\end{equation}

From now on, we set $\Omega = [0,1]^2$ and define the triangular mesh $\mesh$, $h = 1.0, 0.5, 0.25,\ldots$, by first uniformly subdividing $[0,1]^2$ into $4^{n}$ geometrically conforming squares and then subdividing each new square into two right-angled triangles whose hypotenuses connect the south-west and the north-east vertex of the bounding square.
The $n=2$ mesh ($h=0.25$) is depicted in \Cref{fig:mesh}.
In these experiments, we solve~\cref{eq:LL*primalmethodElliptic} using the precisely the same discretization described in \Cref{sub:discretization}.

\begin{figure}
  \begin{tikzpicture} [dot/.style={draw,rectangle,minimum size=4mm,inner sep=0pt,outer sep=0pt,thick}]
  \draw [thin] (0,0) grid (4,4);
  \draw [thin] (0,3) -- (1,4);
  \draw [thin] (0,2) -- (2,4);
  \draw [thin] (0,1) -- (3,4);
  \draw [thin] (0,0) -- (4,4);
  \draw [thin] (1,0) -- (4,3);
  \draw [thin] (2,0) -- (4,2);
  \draw [thin] (3,0) -- (4,1);
\end{tikzpicture}
\caption{\label{fig:mesh}Mesh $\mesh$ with $h=0.25$.}
\end{figure}

Note that \cref{eq:EllipticRegularity} tells us that we can only guarantee that $u,~\lambda \in H^{2+s}(\Omega)$ for $0<s<1$ because each internal angle in $\Omega = [0,1]^2$ is $\pi/2$.
Therefore, by~\Cref{prop:LLstarCorollary,cor:RateLimited}, we can guarantee little more than
\begin{equation}
  \| (\bsigma - \bsigma_{hp},u - u_{hp}) \|_{L^2}
  \leq
  C 
  h^r \big(\|(\bsigma,u)\|_{H^{r}} + \|(\bzeta,\lambda)\|_{H^{r+1}}\big)
  ,
\end{equation}
for $1\leq r < \min\{p+1,2\}$.
Better rates can be uncovered if $u$ is specially designed to give $\lambda$ high regularity.
However, high regularity in $u$ does not imply high regularity in $\lambda$.

\subsection{Results} \label{sub:square_domain_dirichlet}

We consider two seemingly innocuous cases for the loads and boundary conditions: (i) $f(x,y)=2\pi^2\sin(\pi x)\sin(\pi y)$ and $u_0 = 0$; and (ii) $f(x,y)=0$ and  $u_0 = 1$.
In both cases, the exact solution is infinitely smooth.
Indeed, in case (i), $u(x,y)=\sin(\pi x)\sin(\pi y)$ and, in case (ii), $u(x,y)=1$.
\revised{Note that $u\in C^\infty(\overline{\Omega})$ in both cases.}

In the first case (i), a straightforward computation shows that $\lambda$ is a constant scalar multiple of $\sin(\pi x)\sin(\pi y)$ and so $\lambda \in C^\infty(\overline{\Omega})$ is infinitely smooth.
Therefore, by~\cref{prop:LLstarCorollary}, the convergence rate of the $\mcL\mcL^\ast$ method under uniform $h$-refinement will be limited only by polynomial order of the finite element discretization $p$.
This fact is clearly witnessed in~\cref{fig:hUnifA}.

\begin{figure}
  \centering
  \begin{subfigure}[b]{0.65\textwidth}
    \centering
    \includegraphics[width=\textwidth]{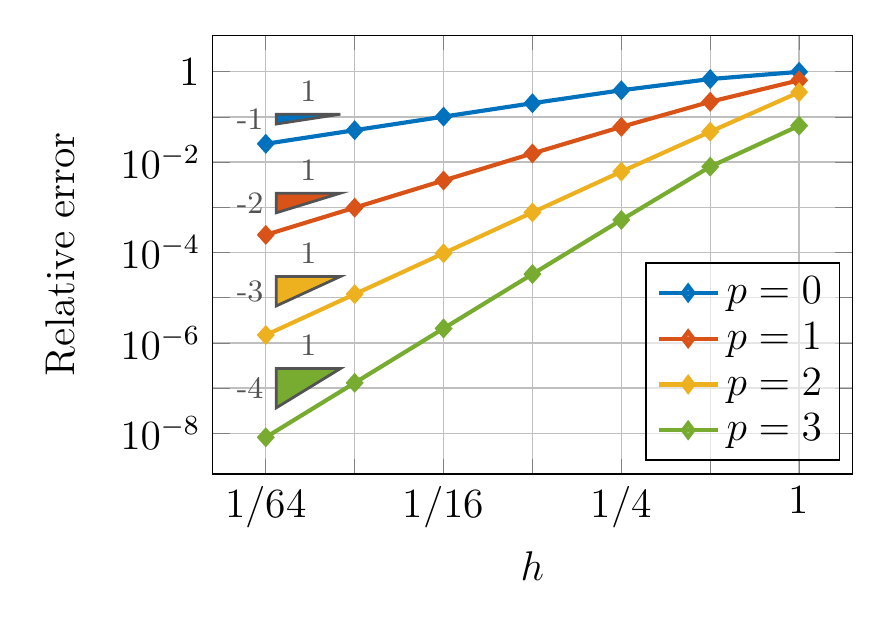}
        \caption{Standard rates; $u(x,y) = \sin(\pi x)\sin(\pi y)$.}
  \label{fig:hUnifA}
  \end{subfigure}  \\[5pt]
  \begin{subfigure}[b]{0.65\textwidth}
    \centering
    \includegraphics[width=\textwidth]{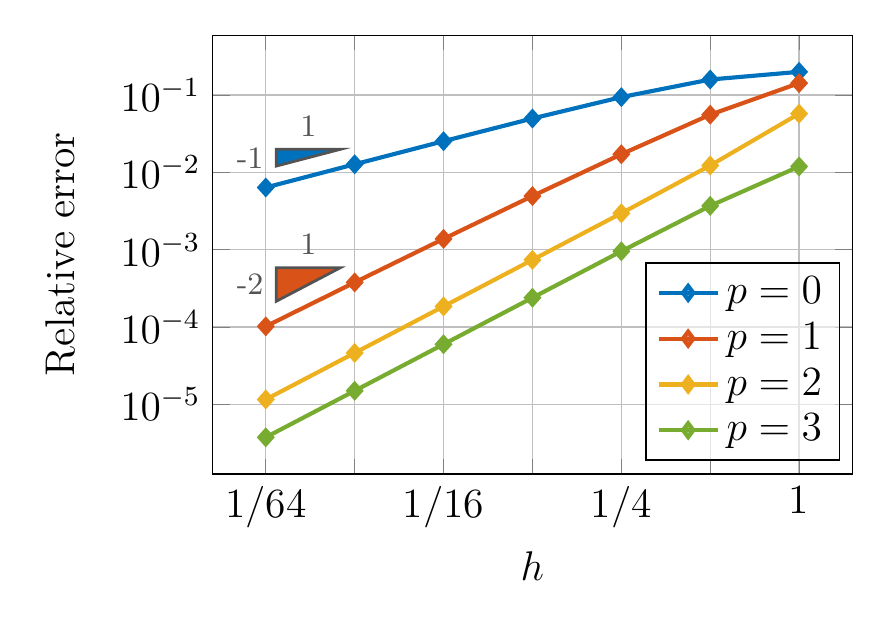}
        \caption{Sub-standard rates; $u(x,y) = 1$.}
  \label{fig:hUnifB}
  \end{subfigure}
  \caption{Different convergence rates for two different $C^\infty(\overline{\Omega})$ exact solutions $u$.
  In~\subref{fig:hUnifA}, $\lambda \in C^\infty(\overline{\Omega})$, however, in~\subref{fig:hUnifB}, $\lambda \in H^{2+s}(\Omega)$ for all $0<s<1$.
  }
    \label{fig:hUnifLLstar}
\end{figure}

In the second case (ii), one sees from~\cref{eq:LambdaPoissonInhomogeneous} that $\lambda$ solves
\begin{equation}
  -\Delta \lambda = 1
  \quad
  \text{in } \Omega,
  \quad
  \text{with }
  \lambda = 0
  \quad
  \text{on } \partial\Omega.
\end{equation}
This problem does not have a simple exact solution to write out, however, we know that $\lambda \in H^{2+s}(\Omega)$ for $0<s<1$.
In turn, in~\cref{fig:hUnifB}, we see at best second-order convergence of the exact solution $(\bsigma,u)$ even though $u\in C^\infty(\overline{\Omega})$.

\begin{remark}
\Cref{fig:hUnifLLstar} reinforces both~\Cref{prop:LLstarCorollary,cor:RateLimited} to demonstrate the $\mcL\mcL^\ast$ method is rate limited.
Even though we considered only the special case $\bmA = \bmI$, $\bmb = (0,0)$, and $c = 0$, similar conclusions also hold for more complicated coefficients.
In \cite[Section~5]{demkowicz2018dpg}, analogous experiments were done with the DPG* method and with an $\mcL\mcL^\ast$ with a tensor product discretization.
Identical conclusions were found in both of those settings as well.
\end{remark}

\begin{remark}
  It is important to note that the rate-limited behavior of the solution only influences the accuracy of high-order discretizations.
    Because most early experiments with minimum norm methods involved only the lowest order ($p=0$) setting, this may help to explain the confusion on this topic in the literature.
\end{remark}

\begin{remark}
  One important consequence of the rate-limited behavior of minimum norm methods can be found in the context of adjoint-based \emph{a posteriori} error estimation with minimum residual methods.
  The dual-weighted residual approach to \emph{a posteriori} error estimation \cite{becker2001optimal}, one of predominant approaches in the literature, requires an adjoint solution with greater accuracy than the primal solution.
  Now, as multiple authors have noticed \cite{fuhrer2017superconvergent,Keith2017Goal,valseth2020goal}, the adjoints of minimum residual methods are, in fact, minimum norm methods.
  Nevertheless, as Valseth \emph{et al.} witnessed in \cite{valseth2020goal}, if one uses a minimum norm method to solve the adjoint problem on the same mesh with higher-order elements, then accuracy may not increase enough and the approach may break down.
\end{remark}

\section{Discussion} \label{sec:discussion}

As mentioned in the introduction, a growing number of non-standard finite element methods with a specific saddle-point structure have appeared in the literature in recent years.
Each of these methods can be characterized by a specific form of PDE-constrained norm minimization problem; cf. \Cref{sub:notation}.
The progenitor of these methods is arguably the $\mcL\mcL^\ast$ (FOSLL*) method introduced by Cai \emph{et al.} in \cite{cai2001first}.

Using the $\mcL\mcL^\ast$ method as an example, we have argued that all standard minimum norm methods are rate limited.
This property arises naturally from the saddle-point structure of these methods, not from any other defining feature in their individual constructions; \emph{e.g.}, discontinuous or solely $H^1$-conforming approximation spaces.
Our conclusions do not say that future research on minimum norm methods is unwarranted.
Indeed, they remain an important class of intrinsically stable finite element methods.

There were several topics of interest not considered in this work because they would be out of scope.
The most important of which is arguably \emph{a posteriori} error estimation and adaptive mesh refinement.
Indeed, a reliable and efficient \emph{a posteriori} error estimator \cite{cai2015div,demkowicz2018dpg} can still be used to drive an adaptive mesh refinement process which will recover optimal convergence rates for all $p$.
Moreover, recent advances on, \emph{e.g.}, energy-corrected methods \cite{egger2014energy} may lead to another possibility to recover optimal convergence rates, even under $h$-uniform mesh refinements.
Another appealing alternative is to further explore the use of weighted norms which have successfully been applied to fully $H^1$-conforming $\mcL\mcL^\ast$ discretizations \cite{manteuffel2005first,lee2007fosll}.
In a different context, weighted norms can also be incorporated in goal-oriented methods; see, e.g., \cite{chaudhry2014enhancing,kergrene2017new}.
We must mention that it is also possible to blend some of the features of minimum norm methods and minimum residual methods; see, \emph{e.g.}, the separate approaches taken \cite{bui2014pde} and \cite{kalchev2018mixed,kalchev2020least}.

This work has allowed us to resolve a contentious confusion in the community at large, with important implications for future research.
One especially important implication is for adjoint-based \emph{a posteriori} error estimation with modern minimum residual methods.
Indeed, if one chooses to apply the dual-weighted residual method \cite{becker2001optimal} \textemdash{} which typically involves solving an adjoint problem with the same mesh but higher-order elements \textemdash{} we can explain why the accuracy of the adjoint solution may not be any better than the accuracy of the primal solution, as pointed out in \cite{valseth2020goal}.

\section*{Acknowledgements} \label{sec:acknowledgements}

I wish to thank Federico Fuentes for helpful discussions and proofreading of the manuscript.
In addition, I extend my sincere gratitude to one anonymous referee who, although advocating for rejection of an earlier version, also provided numerous kind and helpful comments which improved the quality of this work.

The majority of this manuscript was written while the author was in residence at the Institute for Computational and Experimental Research in Mathematics (ICERM) in Providence, RI, during the Advances in Computational Relativity program, supported by the National Science Foundation under Grant No. DMS-1439786.
Final edits were completed while at Lawrence Livermore National Laboratory.

This work was performed under the auspices of the U.S. Department of Energy by Lawrence Livermore National Laboratory under Contract DE-AC52-07NA27344, LLNL-JRNL-826017-DRAFT.

\phantomsection\bibliographystyle{abbrv}
\bibliography{main}

\begin{thebibliography}{10}

\bibitem{alnaes2015fenics}
M.~Aln{\ae}s, J.~Blechta, J.~Hake, A.~Johansson, B.~Kehlet, A.~Logg,
  C.~Richardson, J.~Ring, M.~E. Rognes, and G.~N. Wells.
\newblock The {FEniCS} project version 1.5.
\newblock {\em Archive of Numerical Software}, 3(100), 2015.

\bibitem{bacuta2003regularity}
C.~Bacuta, J.~Bramble, and J.~Xu.
\newblock Regularity estimates for elliptic boundary value problems in besov
  spaces.
\newblock {\em Mathematics of Computation}, 72(244):1577--1595, 2003.

\bibitem{bacuta2003regularity2}
C.~Bacuta, J.~H. Bramble, and J.~Xu.
\newblock Regularity estimates for elliptic boundary value problems with smooth
  data on polygonal domains.
\newblock {\em Journal of Numerical Mathematics}, 11(2):75--94, 2003.

\bibitem{BACUTA20152920}
C.~Bacuta and K.~Qirko.
\newblock A saddle point least squares approach to mixed methods.
\newblock {\em Comput. Math. Appl.}, 70(12):2920 -- 2932, 2015.

\bibitem{bacuta2017saddle}
C.~Bacuta and K.~Qirko.
\newblock A saddle point least squares approach for primal mixed formulations
  of second order pdes.
\newblock {\em Comput. Math. Appl.}, 73(2):173--186, 2017.

\bibitem{becker2001optimal}
R.~Becker and R.~Rannacher.
\newblock An optimal control approach to a posteriori error estimation in
  finite element methods.
\newblock {\em Acta Numer.}, 10:1--102, 2001.

\bibitem{bochev2009least}
P.~B. Bochev and M.~D. Gunzburger.
\newblock {\em Least-squares finite element methods}, volume 166.
\newblock Springer Science \& Business Media, 2009.

\bibitem{Boffi2013}
D.~Boffi, M.~Fortin, and F.~Brezzi.
\newblock {\em Mixed finite element methods and applications}.
\newblock Springer series in computational mathematics. Springer, Berlin,
  Heidelberg, 2013.

\bibitem{brugiapaglia2015compressed}
S.~Brugiapaglia, S.~Micheletti, and S.~Perotto.
\newblock Compressed solving: A numerical approximation technique for elliptic
  {PDEs} based on compressed sensing.
\newblock {\em Comput. Math. Appl.}, 70(6):1306--1335, 2015.

\bibitem{brugiapaglia2018theoretical}
S.~Brugiapaglia, F.~Nobile, S.~Micheletti, and S.~Perotto.
\newblock A theoretical study of {COmpRessed SolvING} for
  advection-diffusion-reaction problems.
\newblock {\em Mathematics of Computation}, 87(309):1--38, 2018.

\bibitem{brugiapaglia2020compressive}
S.~Brugiapaglia, L.~Tamellini, and M.~Tani.
\newblock Compressive isogeometric analysis.
\newblock {\em Comput. Math. Appl.}, 80(12):3137--3155, 2020.

\bibitem{bui2014pde}
T.~Bui-Thanh and O.~Ghattas.
\newblock A {PDE}-constrained optimization approach to the discontinuous
  {Petrov--Galerkin} method with a trust region inexact {Newton-CG} solver.
\newblock {\em Comput. Methods Appl. Mech. Engrg.}, 278:20--40, 2014.

\bibitem{cai2015div}
Z.~Cai, R.~Falgout, and S.~Zhang.
\newblock Div first-order system {LL}*({FOSLL}*) for second-order elliptic
  partial differential equations.
\newblock {\em SIAM Journal on Numerical Analysis}, 53(1):405--420, 2015.

\bibitem{cai1994first}
Z.~Cai, R.~Lazarov, T.~A. Manteuffel, and S.~F. McCormick.
\newblock First-order system least squares for second-order partial
  differential equations: {P}art {I}.
\newblock {\em SIAM J. Numer. Anal.}, 31(6):1785--1799, 1994.

\bibitem{cai1997first}
Z.~Cai, T.~A. Manteuffel, and S.~F. McCormick.
\newblock First-order system least squares for second-order partial
  differential equations: {P}art {II}.
\newblock {\em SIAM J. Numer. Anal.}, 34(2):425--454, 1997.

\bibitem{cai2001first}
Z.~Cai, T.~A. Manteuffel, S.~F. McCormick, and J.~Ruge.
\newblock First-order system $\mathcal{L}\mathcal{L}^\ast$ ({FOSLL}*): Scalar
  elliptic partial differential equations.
\newblock {\em SIAM J. Numer. Anal.}, 39(4):1418--1445, 2001.

\bibitem{calo2020adaptive}
V.~M. Calo, A.~Ern, I.~Muga, and S.~Rojas.
\newblock An adaptive stabilized conforming finite element method via residual
  minimization on dual discontinuous galerkin norms.
\newblock {\em Computer Methods in Applied Mechanics and Engineering},
  363:112891, 2020.

\bibitem{Calo2018automatic}
V.~M. Calo, A.~Romkes, and E.~Valseth.
\newblock Automatic variationally stable analysis for fe computations: An
  introduction.
\newblock In G.~R. Barrenechea and J.~Mackenzie, editors, {\em Boundary and
  Interior Layers, Computational and Asymptotic Methods BAIL 2018}, pages
  19--43, Cham, 2020. Springer International Publishing.

\bibitem{carstensen2014posteriori}
C.~Carstensen, L.~Demkowicz, and J.~Gopalakrishnan.
\newblock A posteriori error control for {DPG} methods.
\newblock {\em SIAM J. Numer. Anal.}, 52(3):1335--1353, 2014.

\bibitem{chakraborty2020optimal}
A.~Chakraborty, A.~Rangarajan, and G.~May.
\newblock Optimal approximation spaces for discontinuous petrov-galerkin finite
  element methods.
\newblock {\em arXiv preprint arXiv:2012.12751}, 2020.

\bibitem{chan2014dual}
J.~Chan, J.~A. Evans, and W.~Qiu.
\newblock A dual {P}etrov--{G}alerkin finite element method for the
  convection-diffusion equation.
\newblock {\em Comput. Math. Appl.}, 68(11):1513--1529, 2014.

\bibitem{chaudhry2014enhancing}
J.~H. Chaudhry, E.~C. Cyr, K.~Liu, T.~A. Manteuffel, L.~N. Olson, and L.~Tang.
\newblock Enhancing least-squares finite element methods through a
  quantity-of-interest.
\newblock {\em SIAM J. Numer. Anal.}, 52(6):3085--3105, 2014.

\bibitem{cohen2012adaptivity}
A.~Cohen, W.~Dahmen, and G.~Welper.
\newblock Adaptivity and variational stabilization for convection-diffusion
  equations.
\newblock {\em ESAIM Math. Model. Numer. Anal.}, 46(5):1247--1273, 2012.

\bibitem{dahmen2012adaptive}
W.~Dahmen, C.~Huang, C.~Schwab, and G.~Welper.
\newblock Adaptive {P}etrov--{G}alerkin methods for first order transport
  equations.
\newblock {\em SIAM J. Numer. Anal.}, 50(5):2420--2445, 2012.

\bibitem{DemkowiczClosedRange}
L.~Demkowicz.
\newblock Various variational formulations and closed range theorem.
\newblock ICES Report 15-03, The University of Texas at Austin, 2015.

\bibitem{demkowicz2010class}
L.~Demkowicz and J.~Gopalakrishnan.
\newblock A class of discontinuous {P}etrov--{G}alerkin methods. {P}art {I}:
  {T}he transport equation.
\newblock {\em Comput. Methods Appl. Mech. Engrg.}, 199(23-24):1558--1572,
  2010.

\bibitem{demkowicz2011class}
L.~Demkowicz and J.~Gopalakrishnan.
\newblock A class of discontinuous {P}etrov--{G}alerkin methods. {II}.
  {O}ptimal test functions.
\newblock {\em Numer. Methods Partial Differ. Equ.}, 27(1):70--105, 2011.

\bibitem{demkowicz2013primal}
L.~Demkowicz and J.~Gopalakrishnan.
\newblock A primal {DPG} method without a first-order reformulation.
\newblock {\em Comput. Math. Appl.}, 66(6):1058--1064, 2013.

\bibitem{demkowicz2015encyclopedia}
L.~Demkowicz and J.~Gopalakrishnan.
\newblock Discontinuous {P}etrov--{G}alerkin ({DPG}) method.
\newblock In E.~Stein, R.~Borst, and T.~J.~R. Hughes, editors, {\em
  Encyclopedia of Computational Mechanics Second Edition}, pages 1--15. Wiley
  Online Library, 2017.

\bibitem{demkowicz2018dpg}
L.~Demkowicz, J.~Gopalakrishnan, and B.~Keith.
\newblock The {DPG}-star method.
\newblock {\em Comput. Math. Appl.}, 79(11):3092 -- 3116, 2020.

\bibitem{demkowicz2012class}
L.~Demkowicz, J.~Gopalakrishnan, and A.~H. Niemi.
\newblock A class of discontinuous {P}etrov--{G}alerkin methods. {P}art {III}:
  {A}daptivity.
\newblock {\em Appl. Numer. Math.}, 62(4):396--427, 2012.

\bibitem{egger2014energy}
H.~Egger, U.~R\"ude, and B.~Wohlmuth.
\newblock Energy-corrected finite element methods for corner singularities.
\newblock {\em SIAM Journal on Numerical Analysis}, 52(1):171--193, 2014.

\bibitem{MR2597943}
L.~C. Evans.
\newblock {\em Partial differential equations}, volume~19 of {\em Graduate
  Studies in Mathematics}.
\newblock American Mathematical Society, Providence, RI, second edition, 2010.

\bibitem{fuhrer2017superconvergence}
T.~F{\"u}hrer.
\newblock Superconvergence in a {DPG} method for an ultra-weak formulation.
\newblock {\em Comput. Math. Appl.}, 75(5):1705--1718, 2018.

\bibitem{fuhrer2017superconvergent}
T.~F{\"u}hrer.
\newblock Superconvergent {DPG} methods for second-order elliptic problems.
\newblock {\em Comput. Meth. Appl. Mat.}, 19(3):483--502, 2019.

\bibitem{Grisvard}
P.~Grisvard.
\newblock {\em Singularities in Boundary Value Problems}.
\newblock Springer-Verlag, Paris, 1992.

\bibitem{grisvard2011elliptic}
P.~Grisvard.
\newblock {\em Elliptic problems in nonsmooth domains}.
\newblock SIAM, 2011.

\bibitem{houston2020eliminating}
P.~Houston, S.~Roggendorf, and K.~G. van~der Zee.
\newblock Eliminating {G}ibbs phenomena: A non-linear {P}etrov--{G}alerkin
  method for the convection--diffusion--reaction equation.
\newblock {\em Comput. Math. Appl.}, 80(5):851--873, 2020.

\bibitem{kalchev2020least}
D.~Z. Kalchev and T.~A. Manteuffel.
\newblock A least-squares finite element method based on the helmholtz
  decomposition for hyperbolic balance laws.
\newblock {\em Numerical Methods for Partial Differential Equations}, 2020.

\bibitem{kalchev2018mixed}
D.~Z. Kalchev, T.~A. Manteuffel, and S.~M{\"u}nzenmaier.
\newblock Mixed and least-squares finite element methods with application to
  linear hyperbolic problems.
\newblock {\em Numerical Linear Algebra with Applications}, 25(3):e2150, 2018.

\bibitem{Keith2018thesis}
B.~Keith.
\newblock {\em New ideas in adjoint methods for {PDE}s: {A} saddle-point
  paradigm for finite element analysis and its role in the {DPG} methodology}.
\newblock PhD thesis, The University of Texas at Austin, Austin, Texas, U.S.A.,
  2018.

\bibitem{Keith2017Goal}
B.~Keith, A.~V. Astaneh, and L.~Demkowicz.
\newblock Goal-oriented adaptive mesh refinement for discontinuous
  {P}etrov--{G}alerkin methods.
\newblock {\em SIAM J. Numer. Anal.}, 57(4):1649--1676, 2019.

\bibitem{Keith2017DPGstarICESReport}
B.~Keith, L.~Demkowicz, and J.~Gopalakrishnan.
\newblock {DPG}* method.
\newblock ICES Report 17-25, The University of Texas at Austin, 2017.

\bibitem{Keith2017Discrete}
B.~Keith, S.~Petrides, F.~Fuentes, and L.~Demkowicz.
\newblock Discrete least-squares finite element methods.
\newblock {\em Comput. Methods Appl. Mech. Engrg.}, 327:226--255, 2017.

\bibitem{kergrene2017new}
K.~Kergrene, S.~Prudhomme, L.~Chamoin, and M.~Laforest.
\newblock A new goal-oriented formulation of the finite element method.
\newblock {\em Computer Methods in Applied Mechanics and Engineering},
  327:256--276, 2017.

\bibitem{kondrat1967boundary}
V.~A. Kondratiev.
\newblock Boundary value problems for elliptic equations in domains with
  conical or angular points.
\newblock {\em Trans. Moscow Math. Soc.}, 16:209--292, 1967.

\bibitem{lee2007fosll}
E.~Lee and T.~A. Manteuffel.
\newblock {FOSLL}* method for the eddy current problem with three-dimensional
  edge singularities.
\newblock {\em SIAM journal on numerical analysis}, 45(2):787--809, 2007.

\bibitem{lee2015fosll}
E.~Lee, T.~A. Manteuffel, and C.~R. Westphal.
\newblock {FOSLL}* for nonlinear partial differential equations.
\newblock {\em SIAM Journal on Scientific Computing}, 37(5):S503--S525, 2015.

\bibitem{los2020isogeometric}
M.~Los, J.~Mu{\~n}oz-Matute, I.~Muga, and M.~Paszynski.
\newblock Isogeometric residual minimization method ({iGRM}) with direction
  splitting for non-stationary advection-diffusion problems.
\newblock {\em Comput. Math. Appl.}, 79(2):213--229, 2020.

\bibitem{manteuffel2005first}
T.~A. Manteuffel, S.~F. McCormick, J.~Ruge, and J.~Schmidt.
\newblock First-order system {LL}*({FOSLL}*) for general scalar elliptic
  problems in the plane.
\newblock {\em SIAM journal on numerical analysis}, 43(5):2098--2120, 2005.

\bibitem{muga2019discrete}
I.~Muga, M.~J. Tyler, and K.~G. van~der Zee.
\newblock The discrete-dual minimal-residual method (ddmres) for weak
  advection-reaction problems in banach spaces.
\newblock {\em Computational Methods in Applied Mathematics}, 19(3):557--579,
  2019.

\bibitem{raviart1977mixed}
P.-A. Raviart and J.-M. Thomas.
\newblock A mixed finite element method for 2-nd order elliptic problems.
\newblock In {\em Mathematical aspects of finite element methods}, pages
  292--315. Springer, 1977.

\bibitem{rojas2020residual}
S.~Rojas, D.~Pardo, P.~Behnoudfar, and V.~M. Calo.
\newblock Residual minimization for goal-oriented adaptivity.
\newblock {\em arXiv preprint arXiv:2007.08824}, 2020.

\bibitem{scott1990finite}
L.~R. Scott and S.~Zhang.
\newblock Finite element interpolation of nonsmooth functions satisfying
  boundary conditions.
\newblock {\em Mathematics of Computation}, 54(190):483--493, 1990.

\bibitem{valseth2020goal}
E.~Valseth and A.~Romkes.
\newblock Goal-oriented error estimation for the automatic variationally stable
  {FE} method for convection-dominated diffusion problems.
\newblock {\em Comput. Math. Appl.}, 80(12):3027--3043, 2020.

\end{thebibliography}

\end{document}